\documentclass[a4paper]{amsart}

\usepackage{amssymb}
\usepackage{amsthm}
\usepackage{amsmath}
\usepackage{amscd}
\usepackage{graphicx}
\usepackage{dsfont}
\usepackage{xcolor}

\newcommand\1{{\bf 1}}
\newcommand\op[1]{\mathop{\rm #1}\nolimits}
\newcommand\C{\mathbb{C}}
\newcommand\R{\mathbb{R}}
\newcommand\de{\delta}
\newcommand\g{\mathfrak{g}}
\newcommand\h{\mathfrak{h}}
\newcommand\La{\Lambda}
\newcommand\m{\mathfrak{m}}
\newcommand\ot{\otimes}
\renewcommand\th{\theta}
\newcommand\vp{\varphi}
\newcommand\we{\wedge}
\theoremstyle{plain}
\newtheorem{theorema}{Theorem}
\newtheorem{thm}{Theorem}
\newtheorem{prop}[theorema]{Proposition}
\newtheorem{lem}{Lemma}
\newtheorem{dfn}[theorema]{Definition}

\begin{document}
\title[Reconstruction from Representations: Jacobi via Cohomology]{Reconstruction from Representations:\\ Jacobi via Cohomology}
\author{B. Kruglikov, H. Winther}
 \address{Department of Mathematics and Statistics, Faculty of Science and Technology,
 UiT the Arctic University of Norway, Troms\o\ 90-37, Norway}
 \email{ boris.kruglikov@uit.no, \quad henrik.winther@uit.no.}
 \maketitle

 \begin{abstract}
A subalgebra of a Lie algebra $\h\subset\g$ determines $\h$-representation $\rho$ on $\m=\g/\h$.
In this note we discuss how to reconstruct $\g$ from $(\h,\m,\rho)$.
In other words, we find all the ingredients for building non-reductive Klein geometries.
The Lie algebra cohomology plays a decisive role here.
 \end{abstract}

\maketitle

%%%%%%%%%%%%%%%%%%%%%
\section*{Introduction}

Let $\g$ be a Lie algebra and $\h\subset\g$ a subalgebra. Let $\m=\g/\h$ be the quotient $\h$-module with representation
$\rho:\h\to\op{End}(\m)$. We address the reconstruction problem for the Lie algebra structure of $\g$
from the data $(\h,\m,\rho)$.

In this note the spaces ($\h,\m$) are finite-dimensional. We do not assume
the existence of an embedding $\m\subset\g$ as a reductive ($\h$-invariant) complement to $\h$,
as such embeddings exist in general only if $H^1(\h,\m^*\ot\h)=0$.

We show that the Lie algebra cohomology\footnote{We recall the definition of the Lie algebra cohomology
in the appendix.} $H^1(\h,\mathbb{V})$ for $\h$-modules $\mathbb{V}$ plays
a key role in the reconstruction.
Parametrizing Lie brackets on $\h\oplus\m$, the Jacobi identity constrains the parameters.
Our cohomological approach allows to single out linear equations,
and significantly reduce the amount of quadratic constraints.

Klein geometries are homogeneous spaces $G/H$, and such have been extensively studied
for reductive subgroups $H$. Our method allows to effectively handle non-reductive Klein geometries
via symbolic computations.

Some other approaches to the reconstruction in the case of filtered algebras
via the deformation technique can be found in \cite{R,FF,CK}.
We also mention Cartan's procedure for construction of homogeneous models 
of a given geometry type \cite{C}.

%%%%%%%%%%%%%%%%%%%%%
\section{The main result}

Since $\h$ is a subalgebra, the bracket $\Lambda^2\h\to\h\subset \g$ is the Lie algebra structure on
$\h$, but since $\m$ is not a reductive complement the other brackets on $\g$ are:
 \begin{align*}
\h\ot\m\to\g=\h\oplus\m,&\quad [h,u]=\vp(h,u)+h\cdot u;\\
\La^2\m\to\g=\h\oplus\m,&\quad [u_1,u_2]=\th_\h(u_1,u_2)+\th_\m(u_1,u_2),
 \end{align*}
where $h\in\h$, $u,u_1,u_2\in\m$, $\vp\in\h^*\ot\m^*\ot\h$, $\th_\h\in\La^2\m^*\ot\h$,
$\th_\m\in\La^2\m^*\ot\m$, and we denote (here and in what follows) $h\cdot u=\rho(h)u$.

The cohomology class of $\vp$ is known from the splitting theory for modules \cite{Gi}.
We interpret it as the complete obstruction to the existence of a reductive complement.

 \begin{prop}
The element $\vp$ is closed in the complex $\La^\bullet\h^*\ot\m^*\ot\h$: $d\vp=0$,
and it changes by an exact element when $\m\subset\g$ varies. Thus $[\vp]\in H^1(\h,\m^*\ot\h)$.
 % is an invariant.
 \end{prop}

 \begin{proof}
The Jacobi identity with 2 arguments from $\h$ is
 $$
[h_1,[h_2,u]]+[h_2,[u,h_1]]+[u,[h_1,h_2]]=0.
 $$
The $\h$-part of this is the representation property of $\rho$, while the $\m$-part yields
%%% beginning: for the draft %%%
{\footnotesize
\begin{gather*}
[h_1,\vp(h_2,u)]+\vp(h_1,h_2\cdot u)-[h_2,\vp(h_1,u)]-\vp(h_2,h_1\cdot u)-\vp([h_1,h_2],u)=0,\ \textrm{i.e.} \\
(h_1\cdot\vp)(h_2,u)-(h_2\cdot\vp)(h_1,u)+\vp([h_1,h_2],u)=0,\ \textrm{i.e.}
 \end{gather*}}\vspace{-12pt}
%%% end: remove in the paper %%%
 $$
d\vp(h_1,h_2)(u)=0\quad\Leftrightarrow\quad \vp\in Z^1(\h,\m^*\ot\h).
 $$
If we change $\m$ to $\tilde\m=\op{graph}(\sigma)=\{\sigma(u)+u\in\g=\h\oplus\m\}$
for $\sigma\in\m^*\ot\h$ then $\vp$ is changed to $\tilde\vp=\vp+d\sigma$, so
$[\vp]\in Z^1(\h,\m^*\ot\h)/B^1(\h,\m^*\ot\h)$.
 \end{proof}

Let us introduce the linear operator $\de$ as the composition
 $$
\Lambda^i\h^*\ot\m^*\ot\h\stackrel{\rho}\longrightarrow\Lambda^i\h^*\ot\m^*\ot\m^*\ot\m
\stackrel{\op{alt}_{2,3}}\longrightarrow\Lambda^i\h^*\ot\La^2\m^*\ot\m
 $$
where $\op{alt}_{2,3}$ is the skew-symmetrization by the corresponding arguments, e.g.
 $$
\delta\vp(h)(u_1,u_2)=\vp(h,u_1)\cdot u_2-\vp(h,u_2)\cdot u_1.
 $$
 \begin{lem}\label{L0}
For the differential $d$ in the complex $\La^\bullet\h^*\ot\m^*\ot\h$
we have: $[\de,d]=0$, whence $d\delta\vp=0$ and $\delta(\vp+d\sigma)=\delta\vp+d\delta\sigma$.
 \end{lem}
%%% beginning: for the draft %%%
{\footnotesize
\begin{proof}
This is because $\delta$ is obtained from an $\h$-homomorphism of the coefficients module.
 \end{proof}}\vspace{-2pt}
%%% end: remove in the paper %%%

Given $\th_\m\in\La^2\m^*\ot\m$ satisfying $\de\vp=d\th_\m$, let us introduce the nonlinear operator
 $$
Q:\h^*\ot\m^*\ot\h\to\h^*\ot\La^2\m^*\ot\h
 $$
by $Q\vp=\op{Sq}_2(\vp,\vp)-\vp\circ(\1_\h\we\th_\m)$, where $\op{Sq}_2$ is the composition
 $$
(\h^*\ot\m^*\ot\h)^{\ot2}=\h^*\ot\m^*\ot\h\ot\h^*\ot\m^*\ot\h
\stackrel{\op{contr}_{3,4}}\longrightarrow
\h^*\ot\m^*\ot\m^*\ot\h\stackrel{\op{alt}_{2,3}}\longrightarrow\h^*\ot\La^2\m^*\ot\h,
 $$
and $\op{contr}_{3,4}$ is the contraction by the corresponding arguments.
In other words,
 $$
Q\vp(h)(u_1,u_2)=\vp(\vp(h,u_1),u_2)-\vp(\vp(h,u_2),u_1)-\vp(h,\th_\m(u_1,u_2)).
 $$
 \begin{lem}\label{L1}
We have $d(Q\vp)=0$, i.e.\ $[Q\vp]\in H^1(\h,\La^2\m^*\ot\h)$.
 \end{lem}

%%% beginning: for the draft %%%
{\footnotesize\
\begin{proof}
This follows by a straightforward computation: Let
 \begin{align*}
\Xi(h_1,h_2,u_1,u_2)=&h_1\cdot\vp(\vp(h_2,u_1),u_2)-\vp(\vp(h_2,h_1\cdot u_1),u_2)-\vp(\vp(h_2,u_1),h_1\cdot u_2),\\
=&(h_1\cdot\vp)(\vp(h_2,u_1),u_2)+\vp((h_1\cdot\vp)(h_2,u_1),u_2).
 \end{align*}
Then using $d\vp=0$ and $d\theta_\m=\delta\vp$ (and $h_1\cdot h_2=[h_1,h_2]$) we get
 \begin{gather*}
d(Q\vp)(h_1,h_2)(u_1,u_2)=
\Xi(h_1,h_2,u_1,u_2)-\Xi(h_1,h_2,u_2,u_1)-\Xi(h_2,h_1,u_1,u_2)+\Xi(h_2,h_1,u_2,u_1)\\
-(h_1\cdot\vp)(h_2,\theta_\m(u_1,u_2)) -\vp(h_2,(h_1\cdot\theta_\m)(u_1,u_2))\\
+(h_2\cdot\vp)(h_1,\theta_\m(u_1,u_2)) +\vp(h_1,(h_2\cdot\theta_\m)(u_1,u_2))\\
+\vp(\vp([h_1,h_2],u_1),u_2)-\vp(\vp([h_1,h_2],u_2),u_1)-\vp([h_1,h_2],\theta_\m(u_1,u_2))\\
=\Xi(h_1,h_2,u_1,u_2)-\Xi(h_1,h_2,u_2,u_1)-\Xi(h_2,h_1,u_1,u_2)+\Xi(h_2,h_1,u_2,u_1)\\
+\vp(h_1,\vp(h_2,u_1)\cdot u_2)-\vp(h_1,\vp(h_2,u_2)\cdot u_1)
-\vp(h_2,\vp(h_1,u_1)\cdot u_2)+\vp(h_2,\vp(h_1,u_2)\cdot u_1)\\
+\vp(\vp([h_1,h_2],u_1),u_2)-\vp(\vp([h_1,h_2],u_2),u_1)\\
=(\vp(h_2,u_1)\cdot\vp)(h_1,u_2)+\vp(\vp(h_2,u_1)\cdot h_1,u_2)+\vp(h_1,\vp(h_2,u_1)\cdot u_2)\\
-(\vp(h_2,u_2)\cdot\vp)(h_1,u_1)-\vp(\vp(h_2,u_2)\cdot h_1,u_1)-\vp(h_1,\vp(h_2,u_2)\cdot u_1)\\
-(\vp(h_1,u_1)\cdot\vp)(h_2,u_2)-\vp(\vp(h_1,u_1)\cdot h_2,u_2)-\vp(h_2,\vp(h_1,u_1)\cdot u_2)\\
+(\vp(h_1,u_2)\cdot\vp)(h_2,u_1)+\vp(\vp(h_1,u_2)\cdot h_2,u_1)+\vp(h_2,\vp(h_1,u_2)\cdot u_1)\\
=\vp(h_2,u_1)\cdot\vp(h_1,u_2)-\vp(h_2,u_2)\cdot\vp(h_1,u_1)-\vp(h_1,u_1)\cdot\vp(h_2,u_2)+\vp(h_1,u_2)\cdot\vp(h_2,u_1)=0.
 \end{gather*}
 \end{proof}}\vspace{-2pt}
%%% end: remove in the paper %%%

Let us define the linear operators $q:\m^*\ot\h\to\h^*\ot\La^2\m^*\ot\h$, $\sigma\mapsto q_\sigma$, and
$p:(\La^2\m^*\ot\m)^\h\to\h^*\ot\La^2\m^*\ot\h$, $\nu\mapsto p_\nu$, by the formulae
 \begin{align*}
q_\sigma(h)(u_1,u_2)=&
\,d\sigma(\vp(h,u_1),u_2)-d\sigma(\vp(h,u_2),u_1)+\vp(d\sigma(h,u_1),u_2)\\
&-\vp(d\sigma(h,u_2),u_1)+d\sigma(d\sigma(h,u_1),u_2)-d\sigma(d\sigma(h,u_2),u_1),\\
&-\vp(h,\delta\sigma(u_1,u_2))-d\sigma(h,\theta_\m(u_1,u_2))-d\sigma(h,\delta\sigma(u_1,u_2));\\
p_\nu(h)(u_1,u_2)=&\,\vp(h,\nu(u_1,u_2)).
 \end{align*}
For $\sigma\in\m^*\ot\h$ and $\vp,\theta_\m$ as above define the elements $\phi_i\in\La^2\m^*\ot\h$, $1\le i\le 4$, so
 \begin{align*}
\phi_1(u_1,u_2)=&
 \,\vp(\sigma(u_1),u_2)-\vp(\sigma(u_2),u_1), &\
\phi_2(u_1,u_2)=&
 \,[\sigma(u_1),\sigma(u_2)], \\
\phi_3(u_1,u_2)=&
 \,\sigma((\sigma(u_1)\cdot u_2-(\sigma(u_2)\cdot u_1), &\
\phi_4(u_1,u_2)=&
 \,\sigma(\theta_\m(u_1,u_2).
 \end{align*}

 \begin{lem}\label{L2}
We have $d\,p_\nu=0$ and $q_\sigma=d(\phi_1+\phi_2-\phi_3-\phi_4)$, whence $d\,q_\sigma=0$.
% , whence $\op{Im}(q_\sigma),\op{Im}(p_\nu)\subset Z^1(\h,\La^2\m^*\ot\h)$.
 \end{lem}

%%% beginning: for the draft %%%
{\footnotesize\
\begin{proof}
Here the computations are a bit more involved:
 \begin{align*}
d\phi_1(h)(u_1,u_2)=& \,(h\cdot\vp)(\sigma(u_1),u_2)+\vp(h\cdot\sigma(u_1),u_2)-\vp(\sigma(h\cdot u_1),u_2)\\
-\!&\,(h\cdot\vp)(\sigma(u_2),u_1)-\vp(h\cdot\sigma(u_2),u_1)+\vp(\sigma(h\cdot u_2),u_1)\\
=& \,(\sigma(u_1)\cdot\vp)(h,u_2)-(\sigma(u_2)\cdot\vp)(h,u_1)-\vp(\sigma(h\cdot u_1),u_2)+\vp(\sigma(h\cdot u_2),u_1),\\
d\phi_2(h)(u_1,u_2)=& \,h\cdot(\sigma(u_1)\cdot\sigma(u_2))-\sigma(h\cdot u_1)\cdot\sigma(u_2)+\sigma(h\cdot u_2)\cdot\sigma(u_1),\\
d\phi_3(h)(u_1,u_2)=& \,h\cdot\sigma(\sigma(u_1)\cdot u_2)-\sigma(\sigma(h\cdot u_1)\cdot u_2)-\sigma(\sigma(u_1)\cdot(h\cdot u_2))\\
-\!&\,h\cdot\sigma(\sigma(u_2)\cdot u_1)+\sigma(\sigma(h\cdot u_2)\cdot u_1)+\sigma(\sigma(u_2)\cdot(h\cdot u_1)).
 \end{align*}
Therefore (we again use $d\vp=0$ and $d\theta_\m=\delta\vp$ and also the Jacobi identity)
 \begin{gather*}
q_\sigma(h)(u_1,u_2)=
\vp(h,u_1)\cdot\sigma(u_2)-\sigma(\vp(h,u_1)\cdot u_2)-\vp(h,u_2)\cdot\sigma(u_1)+\sigma(\vp(h,u_2)\cdot u_1)\\
+\vp(h\cdot\sigma(u_1),u_2)-\vp(\sigma(h\cdot u_1),u_2)-\vp(h\cdot\sigma(u_2),u_1)+\vp(\sigma(h\cdot u_2),u_1)\\
+(h\cdot\sigma(u_1))\cdot\sigma(u_2)-\sigma(h\cdot u_1)\cdot\sigma(u_2)
-\sigma((h\cdot\sigma(u_1))\cdot u_2)+\sigma(\sigma(h\cdot u_1)\cdot u_2)\\
-(h\cdot\sigma(u_2))\cdot\sigma(u_1)+\sigma(h\cdot u_2)\cdot\sigma(u_1)
+\sigma((h\cdot\sigma(u_2))\cdot u_1)-\sigma(\sigma(h\cdot u_2)\cdot u_1)\\
-\vp(h,\sigma(u_1)\cdot u_2)+\vp(h,\sigma(u_2)\cdot u_1)-h\cdot\sigma(\theta_\m(u_1,u_2))+\sigma(h\cdot\theta_\m(u_1,u_2))\\
-h\cdot\sigma(\sigma(u_1)\cdot u_2)+\sigma(h\cdot(\sigma(u_1)\cdot u_2))
+h\cdot\sigma(\sigma(u_2)\cdot u_1)-\sigma(h\cdot(\sigma(u_2)\cdot u_1))\\
 =
[\vp(h,u_1)\cdot\sigma(u_2)-\vp(h,u_2)\cdot\sigma(u_1)+\vp(h\cdot\sigma(u_1),u_2)-\vp(\sigma(h\cdot u_1),u_2)\\
-\vp(h\cdot\sigma(u_2),u_1)+\vp(\sigma(h\cdot u_2),u_1)-\vp(h,\sigma(u_1)\cdot u_2)+\vp(h,\sigma(u_2)\cdot u_1)]\\
 +
[(h\cdot\sigma(u_1))\cdot\sigma(u_2)-(h\cdot\sigma(u_2))\cdot\sigma(u_1)
-\sigma(h\cdot u_1)\cdot\sigma(u_2)+\sigma(h\cdot u_2)\cdot\sigma(u_1)]\\
 -
[h\cdot\sigma(\sigma(u_1)\cdot u_2)-h\cdot\sigma(\sigma(u_2)\cdot u_1)
-\sigma(\sigma(h\cdot u_1)\cdot u_2)+\sigma(\sigma(h\cdot u_2)\cdot u_1)\\
+\sigma((h\cdot\sigma(u_1))\cdot u_2)-\sigma(h\cdot(\sigma(u_1)\cdot u_2))
-\sigma((h\cdot\sigma(u_2))\cdot u_1)+\sigma(h\cdot(\sigma(u_2)\cdot u_1))]\\
 -
[\sigma(\vp(h,u_1)\cdot u_2)-\sigma(\vp(h,u_2)\cdot u_1)+(h\cdot(\sigma\circ \theta_\m))(u_1,u_2)-\sigma((h\cdot\theta_\m)(u_1,u_2))]\\
=d\phi_1(h)(u_1,u_2)+d\phi_2(h)(u_1,u_2)-d\phi_3(h)(u_1,u_2)-d\phi_4(h)(u_1,u_2).
 \end{gather*}
The other computation is easy:
 \begin{align*}
d(p_\nu)(h_1,h_2)(u_1,u_2)&=d\vp(h_1,h_2)(\nu(u_1,u_2))\\
&-\vp(h_1)((h_2\cdot\nu)(u_1,u_2))+\vp(h_2)((h_1\cdot\nu)(u_1,u_2))=0.
 \end{align*}
 \end{proof}}\vspace{-2pt}
%%% end: remove in the paper %%%

Thus $\op{Im}(q_\sigma)\subset B^1(\h,\La^2\m^*\ot\h)$, $\op{Im}(p_\nu)\subset Z^1(\h,\La^2\m^*\ot\h)$.
Let us denote $\Pi_\vp=\op{Im}(p_\nu)\,\op{mod}B^1(\h,\La^2\m^*\ot\h)\subset H^1(\h,\La^2\m^*\ot\h)$.

We are now ready to state our main result.

 \begin{thm}\label{Thm}
The Jacobi identity $\op{Jac}(v_1,v_2,v_3)=0$ with 1 argument from $\h$ and
the others from $\m$ gives the following constraints on the cohomology $[\vp]\in H^1(\h,\m^*\ot\h)$:
(1) $[\de\vp]=0\in H^1(\h,\La^2\m^*\ot\m)$,  whence $\de\vp=d\th_\m$;\\
(2) $[Q\vp]\equiv0\in H^1(\h,\La^2\m^*\ot\h)\,\op{mod}\Pi_\vp$, so $Q\vp=d\th_\h$ for some choices of $\vp,\theta_\m$.
 \end{thm}

 \begin{proof}
Consider the Jacobi identity with 1 argument from $\h$:
 $$
[h,[u_1,u_2]]+[u_1,[u_2,h]]+[u_2,[h,u_1]]=0.
 $$
Taking $\m$-part of this identity (this is canonical: projection along $\h$), we obtain
 $$
\de\vp(h)(u_1,u_2)=(h\cdot\th_\m)(u_1,u_2)\quad\Leftrightarrow\quad \de\vp
=d\th_\m\in B^1(\h,\La^2\m^*\ot\m).
 $$
This implies (1).

Two remarks are in order. First, changing $\vp\mapsto\vp+d\sigma$ we do not alter the property $[\de\vp]=0$.
Second, $\th_\m$ is determined by constraint (1) modulo $(\La^2\m^*\ot\m)^\h$.
This changes $Q\vp$ by an element $p_\nu\in Z^1(\h,\La^2\m^*\ot\m)$ due to Lemma \ref{L2}.

To obtain (2) consider the $\h$-part of the Jacobi identity.
Note that if we change $\vp\mapsto\vp+d\sigma$, then $\theta_\m\mapsto\theta_\m+\delta\sigma$,
so $Q\vp\mapsto Q\vp+q_\sigma$. This leaves $Q\vp$ closed by Lemma \ref{L1} and does not change the
cohomology class by Lemma \ref{L2}. However, the latter is influenced by the change of $\theta_\m$
as indicated above: $[Q\vp]\mapsto [Q\vp]+[p_\nu]$.

The $\h$-part of the Jacobi identity can now be written for some $p_\nu\in(\La^2\m^*\ot\m)^\h$
(the freedom in a choice of $\theta_\m$) as
 $$
Q\vp+p_\nu=d\theta_\h \quad\Leftrightarrow\quad [Q\vp]\equiv0\,\op{mod}\Pi_\vp.
 $$
This implies (2).
 \end{proof}

Now the reconstruction algorithm from the data $(\h,\m,\rho)$ is the following:
 \begin{itemize}
\item Compute $H^1(\h,\mathbb{V})$ for $\mathbb{V}=\m^*\ot\h, \La^2\m^*\ot\m, \La^2\m^*\ot\h$.
\item Constrain the Lie bracket parameters by (1) and (2) in the theorem.
\item Constrain them further by the quadratic relations $\op{Jac}_\m:\Lambda^3\m\to\h\oplus\m$ so:
 \begin{gather*}
\mathfrak{S}\,\bigl[\vp(\theta_\h(u_1,u_2),u_3)+ \theta_\h(\theta_\m)(u_1,u_2),u_3)\bigr]=0,\\
\mathfrak{S}\,\bigl[\theta_\h(u_1,u_2)\cdot u_3+ \theta_\m(\theta_\m)(u_1,u_2),u_3)\bigr]=0,
 \end{gather*}
where $\mathfrak{S}$ is the cyclic summation by indices $1,2,3$.
 \end{itemize}

%%%%%%%%%%%%%%%%%%%%%
\section{Some specifications}

In this section we discuss computation of the cohomology involved in the constraints of the theorem in several cases.

\smallskip

{\bf 1: Reductive isotropy.} Let us first note that if $\h$ is a semi-simple Lie algebra then
$H^1(\h,\mathbb{V})=0$ for any $\h$-module $\mathbb{V}$ (Whitehead's lemma \cite{SL}).
The same holds true if $\h$ is reductive with the center action on $\mathbb{V}$ being semi-simple
and with no linear invariants.
Then choosing $\vp=0$ the conditions of the theorem equivalently mean that
 $$
\theta_\h:\Lambda^2\m\to\h,\ \theta_\m:\Lambda^2\m\to\m
 $$
are $\h$-equivariant maps. The parameters of these maps are constrained by the
quadratic conditions $\op{Jac}_\m=0$.
Examples of reconstruction of $\g$ in the case of $\h=\mathfrak{sl}(2),\mathfrak{su}(2)$
and $\dim\m=6$ can be found in \cite{AKW}.

\smallskip

{\bf 2: Internal gradings.} Notice that if $\h$ is graded $\h=\oplus_\alpha\h_\alpha$ with $\h_0$ Abelian and $\mathbb{V}$ is a graded $\h$-module $\mathbb{V}=\oplus\mathbb{V}_\alpha$, then \cite{FF}
the cohomology $H^\bullet(\h,\mathbb{V})$ of the complex $\Lambda^\bullet\h^*\ot\mathbb{V}$
coincides with that of the subcomplex $(\Lambda^\bullet\h^*\ot\mathbb{V})_0$ of total degree zero cochains,
and the same is true if we use multi-grading given by $\h$ (i.e.\ $\alpha$ takes values in a
multi-dimensional lattice).

In particular, if $\h$ is Abelian, and $\mathbb{V}$ completely reducible, then $H^1(\h,\mathbb{V})=\h^*\ot\mathbb{V}_0$,
where $\mathbb{V}_0=\mathbb{V}^\h$ is the trivial component. Indeed, for any reductive Lie algebra $\h$ with the center
acting semi-simply on $\mathbb{V}$ we have $H^\bullet(\h,\mathbb{V})=H^\bullet(\h)\otimes\mathbb{V}^\h$ \cite{FF} and if
$\h$ is Abelian, then the differential of its cohomology complex vanishes.

\smallskip

{\bf 3: One-dimensional space of splittings.} In the case $\dim H^1(\h,\m^*\ot\h)=1$
we can normalize $[\vp]=0\vee1$. In the case $[\vp]=0$ an $\h$-invariant splitting $\g=\h\oplus\m$ exists,
and the solutions $(\theta_\m,\theta_\h)$ of the theorem can be found as $\h$-equivariant elements
$(\La^2\m^*\ot\g)^\h$. If $[\vp]=1$, then the solutions space to the constraints of Theorem \ref{Thm} is
an affine space modelled on $(\La^2\m^*\ot\g)^\h$. Indeed, the system of equations
$d(\theta_\m,\theta_\h)=(\delta\vp,Q\vp)$ is linear inhomogeneous in the unknowns $(\theta_\m,\theta_\h)$.

With one solution (often) a-priori known it is easy to parametrize all solutions (to the constraints of Theorem
\ref{Thm}; these parameters are yet subject to the Jacobi constraints with 3 arguments from $\m$).

\smallskip

{\bf 4: Cartan subalgebra.} Let $\h$ be a Cartan subalgebra of a semi-simple Lie algebra $\g$, and $\m=\g/\h$.
 % We can assume for simplicity of exposition that everything is over $\C$, though the conclusion extends to the real
 % case (because complexification commutes with the cohomology functor).
Since $\m_0=\m^\h=0$ we have $H^\bullet(\h,\m^*\ot\h)=0$. In particular, there is a reductive complement.
For the split real form $\g$, we can take the sum of root spaces $\m=\sum_\alpha\R\cdot e_\alpha$ as such.

The other cohomologies involved in Theorem \ref{Thm} are non-trivial (this will have an implication in the
next section): $H^\bullet(\h,\La^2\m^*\ot\m)=\La^\bullet\h^*\ot(\La^2\m^*\ot\m)_0$
(with $0$ referring to the multi-grading induced by a set of simple roots;
if the real form $\g$ is not split, the complexification can be used at this step), and similarly
$H^\bullet(\h,\La^2\m^*\ot\h)=\La^\bullet\h^*\ot(\La^2\m^*)_0\ot\h$.

Note that the submodule $(\La^2\m^*)_0$ is generated by the elements
$e_\alpha\otimes\theta^\alpha$, where $\theta^\alpha$ is the co-basis dual to the basis $e_\alpha$ of $\m$, and similarly
the submodule $(\La^2\m^*\ot\m)_0$ is generated by the elements $e_{\alpha+\beta}\otimes\theta^\alpha\we\theta^\beta$.
Thus both submodules and hence the cohomology groups are non-trivial.

\smallskip

{\bf 5: Nilradical of a parabolic.} Let $\g=\oplus\g_i$ be a semi-simple Lie algebra
with the grading induced by the parabolic subalgebra $\mathfrak{p}=\oplus_{i\ge0}\g_i$.
We choose as subalgebra the nilradical of the opposite parabolic: $\h=\oplus_{i<0}\g_i$,
and $\m=\g/\h=\mathfrak{p}$.

The cohomology $H^\bullet(\h,\mathbb{V})$ for $\g$-modules $\mathbb{V}$ (restricted to $\h\subset\g$) is given
(as a $\g_0$-module) by Kostant's version of the Bott-Borel-Weyl theorem \cite{Kos},
however in the case $\m=\mathfrak{p}$ it is not a $\g$-module and the computations are more involved.

We compute the cohomology $H^1(\h,\m^*\ot\h)$ in the case of $|1|$-gradings\footnote{The $|1|$-gradings
of simple complex Lie algebras are: $A_\ell/P_k$, $B_\ell/P_1$, $C_\ell/P_\ell$, $D_\ell/P_1$, $D_\ell/P_\ell$, $E_6/P_6$,
$E_7/P_7$ in Bourbaki's ordering of the nodes on the Dynkin diagram. Which of those extend to the real versions
is decided by the Satake diagram.}:
$\g=\h\oplus\g_0\oplus\h^*$ (where $\g_1$ is identified with $\h^*=\g_{-1}^*$ via the Killing form,
and similarly $\g_0^*=\g_0$). We have: $\m=\g_0\oplus\h^*$, $\m^*=\g_0\oplus\h$
(this decomposition of modules is not $\h$-invariant; to get an invariant computation one
should use the spectral sequences based on the $\h$-invariant filtration, but we skip doing so).

The cochain complex for the cohomology is
 $$
0\to(\g_0\oplus\h)\ot\h\stackrel{d_0}\longrightarrow\h^*\ot(\g_0\oplus\h)\ot\h\stackrel{d_1}\longrightarrow
\La^2\h^*\ot(\g_0\oplus\h)\ot\h\to\dots
 $$
where $d_0$ projects the first term to $\g_0\ot\h\subset\h^*\ot\h\ot\h$ and
$d_1$ projects the second term to $\h^*\ot\g_0\ot\h\stackrel{\delta\ot\op{id}_\h}\longrightarrow\La^2\h^*\ot\h\ot\h$,
with $\delta:\h^*\ot\g_0\subset\h^*\ot\h^*\ot\h\to\La^2\h^*\ot\h$ being the Spencer operator
(skew-symmetrization). Thus $B^1=\g_0\ot\h\subset\h^*\ot\h\ot\h$, $Z^1=\g_0^{(1)}\ot\h\oplus\h^*\ot\h\ot\h$,
where $\g_0^{(1)}=\op{Ker}(\delta)=\g_0\ot\h\cap S^2\h^*\ot\h$ is the Sternberg-Spencer prolongation of $\g_0$, whence
$H^1=Z^1/B^1=[\g_0^{(1)}+\h^*\ot\h/\g_0]\ot\h$.

By Yamaguchi's prolongation theorem \cite{Y} we have $\g^{(1)}=\g_1=\h^*$ in all $|1|$-graded cases except $A_\ell/P_1$
(and dually $A_\ell/P_\ell$). In the latter case $\g_0=\mathfrak{gl}(\h)$, so $\g^{(1)}=S^2\h^*\ot\h$.
Thus we conclude
 $$
H^1(\h,\m^*\ot\h)=\left\{\begin{aligned}
& S^2\h^*\ot\h\ot\h & \text{when the grading is }A_\ell/P_1,\\
& \left[\h^*+\frac{\h^*\ot\h}{\g_0}\right]\ot\h & \text{otherwise}.
\end{aligned}\right.
 $$

%%%%%%%%%%%%%%%%%%%%%
\section{Applications}

In this section we consider some simple applications, illustrating the developed technique.
More substantial outcomes can be extracted from \cite{KW1,KW2}
(the first reference is in retrospective, while the second essentially uses the results of this paper).
An application to reconstruction from another technique can be found in \cite{FS}.

\medskip

We start with a toy example: one-dimensional subalgebra $\h\subset\mathfrak{sl}_2$.
Let $\m=\mathfrak{sl}_2/\h$ and $\rho$ be the isotropy representation.
If $\h$ is a (non-compact) Cartan subalgebra, then the triple $(\h,\m,\rho)$ recovers either the original
Lie algebra $\mathfrak{sl}_2$ or the algebra $\R\ltimes\R^2$, where the action of $\h=\R$ on the Abelian
piece $\R^2$ is by the matrix $\op{diag}(-1,1)$.
If, on the other hand, $\h$ is nilpotent, then the triple $(\h,\m,\rho)$ recovers either the original
semi-simple Lie algebra or a solvable algebra $\g$ with $[\g,\g]\supset\h$.

Now if $\h$ is two-dimensional (Borel), then $(\h,\m,\rho)$ recovers either the original $\g=\mathfrak{sl}_2$
or the above $\R\ltimes\R^2=\langle e_1,e_2,e_3:[e_1,e_2]=e_2,[e_1,e_3]=-e_3]\rangle$, $\h=\langle e_1,e_2\rangle$.
Note that for the corresponding homogeneous space $G/H$ the isotropy has a kernel: in the Lie
subalgebra $\h$ the element $e_2$ generates an ideal and thus acts non-effectively.
As a result we obtain an effective homogeneous representation $G/H=G'/H'$ with Lie algebras corresponding to the groups:
$\g'=\langle e_1,e_3\rangle$, $\h=\langle e_1\rangle$.

This motivates the following

 \begin{dfn}
We call a pair $(\h,\g)$ reconstruction rigid, if $\g$ is the unique non-flat algebra with
the isotropy data $(\h,\m,\rho)$ and no nontrivial $\g$-ideals supported in $\h$.
 \end{dfn}

Recall that the flat algebra is $\g=\h\ltimes\m$, where $\m$ is Abelian and the bracket $\h\we\m\to\m$
is $\rho$. Now we consider three examples.

\smallskip

{\bf 1: Borel subalgebra in $\mathfrak{sl}_3$.}
Let $\h=\mathfrak{b}$ be a minimal parabolic (maximal solvable) subalgebra in $\mathfrak{sl}_3$
(realized as the set of upper-triangular matrices) and $\m=\mathfrak{sl}_3/\h$.
The cohomology group $H^1(\h,\m^*\ot\h)=\R$, so there exists a unique (one-dimensional) obstruction to finding
a reductive complement to $\h$. If the obstruction vanishes, then the Lie algebra structure on $\h\oplus\m$
is encoded by $(\La^2\m^*\ot(\h\oplus\m))^\h=0$. So this case is flat: the only possible Lie algebra structure is
$\g=\mathfrak{b}\ltimes\R^3$.

If the obstruction $[\vp]\in\R$ is non-zero, it can be normalized to 1.
In this case we compute the other cohomologies\footnote{In this and subsequent computations
the \textsl{DifferentialGeometry} package of \textsc{Maple} was used.}:
$H^1(\h,\La^2\m^*\ot\m)=\R^3$, $H^1(\h,\La^2\m^*\ot\h)=0$.
The first constraint of our theorem, $\delta\vp=d\theta_\m$, is non-trivial and locks all the structure
constants to be equal to that of $\mathfrak{sl}_3$.
Thus $(\mathfrak{b},\mathfrak{sl}_3)$ is reconstruction rigid.
This shows the uniqueness of (effective) homogeneous representation with 5D isotropy of the flag variety $F(1,2,3)=\op{SL}_3(\R)/B=O(3)/\mathbb{Z}_2^3$
(this Klein geometry is non-reductive unless we reduce the symmetry group: $F(1,2,3)=O(3)/\mathbb{Z}_2^3$).

Note that the isotropy data of the homogeneous space $M^3=G/H$ include the invariant contact 2-distribution
(it is integrable in the flat case) that is split as a sum of two line fields $\Pi^2=L^1_1\oplus L^1_2\subset TM$.
This geometric structure corresponds to a second order ODE with respect to point transformations.
Our computation shows that there exists a unique ODE with symmetry of the maximal dimension 8
(this is indeed $y''(x)=0$ viewed as a pair of line fields -- vertical and the total derivative --
on $J^1(\R,\R)=M^3(x,y,y')$). In fact, it is known that the other homogeneous geometries encoding second order
ODEs can exist on 3-dimensional Lie groups only (trivial isotropy).
Our approach allows us to independently verify this.

\medskip

{\bf 2: Cartan subalgebra in $\mathfrak{sl}_3$.}
Let $\h=\mathfrak{c}$ be a Cartan subalgebra in $\mathfrak{sl}_3$
(realized as the set of diagonal matrices) and $\m=\mathfrak{sl}_3/\h$.
Since we know from the previous section that $H^1(\h,\m^*\ot\h)=0$ and $\m\subset\g$ can be $\h$-invariant,
the reconstruction is reduced to classifying $\h$-equivariant maps $\La^2\m\to\g=\h\oplus\m$.
A simple computation shows that the results of reconstruction is as follows:
 \begin{enumerate}
\item $\g=\mathfrak{sl}_3$, $\h=\mathfrak{c}$;
\item $\g=\mathfrak{gl}_2\ltimes W$, $\h\subset\mathfrak{gl}_2$ is the diagonal subalgebra, $W$ is Abelian
and is $\R^2\oplus\R^{2*}$ as $\mathfrak{gl}_2$-representation;
\item $\g=\R^2\ltimes W$, where $W=\R^3+\La^2\R^3$ is the 2-step nilpotent algebra with the derived series
of dimensions $(6,3,0)$;
\item $\g=\R^2\ltimes W$, where $W=((\R\oplus\R^2)+\R\otimes\R^2)\oplus\R$ and the first summand is
the 2-step nilpotent algebra with the derived series of dimensions $(5,2,0)$;
\item $\g=\R^2\ltimes W$, where $W=\mathfrak{heis}_3\oplus\mathfrak{heis}_3$ is the sum of two Heisenberg algebras;
\item $\g=\R^2\ltimes W$, where $W=\mathfrak{heis}_3\oplus\R^3$;
\item $\g=\R^2\ltimes W$, where $W=\R^6$ is the Abelian algebra (this is the flat case).
 \end{enumerate}
In cases (3)-(7) the action of $\h=\R^2$ on $W$ is chosen to match the isotropy data.

Thus $(\mathfrak{c},\mathfrak{sl}_3)$ is not reconstruction rigid, and the real version $\op{SL}_3(\R)/(\R^\times)^2$
of the complex flag variety $F^\C(1,2,3)=\op{SL}_3(\C)/B^\C=U(3)/T^3=SU(3)/T^2$ is not (uniquely) recoverable
from its isotropy data.

\medskip

{\bf 3: Homogeneous 4D almost complex spaces with $\op{Sol}_2$ isotropy.}
Let us investigate 4-dimensional almost complex homogeneous spaces with 2-dimensional solvable (non-abelian) isotropy,
i.e.\ $\h$ preserves a complex structure on $\m$. We assume the isotropy representation $\rho$ to be effective.
If the almost complex structure $J$ is non-integrable, then the isotropy representation
$\rho:\mathfrak{sol}_2\to\mathfrak{gl}_2(\C)$ also leaves invariant the Nijenhuis tensor
$0\neq N_J\in\La^2_\C\C^{2*}\ot_{\bar\C}\C^2$: $\rho(v)\cdot N_J=0$ $\forall v\in\mathfrak{sol}_2$.
This uniquely normalizes the representation in some complex basis of $\C^2$ so:
 $$
e_1=\begin{pmatrix}0&0\\0&1\end{pmatrix},\ e_2=\begin{pmatrix}0&0\\1&0\end{pmatrix};\quad
[e_1,e_2]=e_2.
 $$
Now we compute the cohomology:
 $$
H^1(\h,\m^*\ot\h)=\R^2,\ H^1(\h,\La^2\m^*\ot\m)=\R^8,\ H^1(\h,\La^2\m^*\ot\h)=\R^4.
 $$
Thus a-priori we get a non-reductive decomposition and the bracket $\h\we\m\to\g=\h\oplus\m$
takes values in both summands, being parametrized by 2 real numbers. These two
parameters are however forced to vanish by the further homological constraints of Theorem \ref{Thm}
(to be precise, by $[\delta\vp]=0\in H^1(\h,\La^2\m^*\ot\m)$ forming 8 scalar linear equations).
Thus a-posteriori the $\h$-module $\g$ splits, and the remaining brackets
can be chosen as elements of the module $(\La^2\m^*\ot\g)^\h$
satisfying the Jacobi identity.

These are computable, however another look at the Jacobi identity shows
that the complement $\m\subset\g$ can be chosen in such a way that both $\h$ and $\m$ are subalgebras,
though the bracket $[\h,\m]$ takes values in the whole $\g$. This still allows to conclude that the
almost complex structure $J$ on the homogeneous space $M=G/H$ coincides with the left-invariant structure on
some Lie group $M^4$. Such groups are classified, but we do not include this rather long list.

Finally notice that by \cite{Kr} a symmetry of a (non-integrable) almost complex structure $J$ on
a 4-manifold $M$ is at most 4D (in this case $J$ is the left-invariant structure on a Lie group)
unless the 2-distribution $\op{Im}(N_J)\subset TM$ is integrable and $J,N_J$ are projectible along its leaves.
We may also observe this from the obtained structure equations of $\h\oplus\m$ in our case.

\appendix
%%%%%%%%%%%%%%%%%%%%%
\section{Lie algebra cohomology}

For completeness recall the formula for the Lie algebra differential in the complex $\Lambda^\bullet\h^*\otimes\mathbb{V}$ for $H^\bullet(\h,\mathbb{V})$. If
$\vp\in\Lambda^k\h^*\otimes\mathbb{V}$ and $h_i\in\h$, then
 \begin{align*}
d\vp(h_0,\dots,h_k)=\sum&(-1)^i h_i\cdot\vp(h_0,\dots,\check{h}_i,\dots,h_k)\\
+& (-1)^{i+j}\vp([h_i,h_j],h_0,\dots,\check{h}_i,\dots,\check{h}_j,\dots,h_k).
 \end{align*}
Note that if $d\vp=0$, then $h\cdot\vp=d\,i_h\vp$, where $i_h\vp=\vp(h,\cdot)\in\Lambda^{k-1}\h^*\otimes\mathbb{V}$
is the result of substitution of $h$ as the first argument. Thus
$H^\bullet(\h,\mathbb{V})=H^\bullet(\h,\mathbb{V})^\h$ and the equivariancy (of e.g.\ $[\vp]$)
is not a constraint on the cohomology classes.

%%%%%%%%%%%%%%%%%%%%%%%%%%%%%%%%%%%%%%%%%%%%%%%%%%%%%%%%%%%%%%%%%%%%%%%%%%%%

\end{document}